\documentclass{amsart}

\theoremstyle{plain}
\newtheorem{theorem}{Theorem}[section]
\newtheorem{cor}[theorem]{Corollary}
\newtheorem{lem}[theorem]{Lemma}
\newtheorem{pro}[theorem]{Proposition}
\newtheorem{rem}[theorem]{Remark}

\newtheorem*{1}{Theorem}

\numberwithin{equation}{section}

\newcommand{\C}{\mathbb{C}}
\newcommand{\R}{\mathbb{R}}

\newcommand{\Lcal}{\mathcal{L}}
\newcommand{\Jcal}{\mathcal{J}}
\newcommand{\Hcal}{\mathcal{H}}
\newcommand{\Tcal}{\mathcal{T}}
\newcommand{\Mcal}{\mathcal{M}}

\newcommand{\WP}{Weil-Petersson}
\newcommand{\WPm}{Weil-Petersson metric}
\newcommand{\WPg}{Weil-Petersson geodesic}
\newcommand{\hym}{hyperbolic metric} 
\newcommand{\km}{K\"{a}hler metric} 
\newcommand{\kf}{K\"{a}hler form} 
\newcommand{\kp}{K\"{a}hler potential} 
\newcommand{\qf}{quasi-Fuchsian}
\newcommand{\TS}{Teichm\"{u}ller space}

\newcommand{\Tt}{Teichm\"{u}ller theory}
\newcommand{\Tc}{Teichm\"{u}ller curve}
\newcommand{\hqd}{holomorphic quadratic differential}
\newcommand{\RS}{Riemann surface}
\newcommand{\hmp}{harmonic map}
\newcommand{\maxp}{maximum principle}
\newcommand{\Sc}{sectional curvature}
\newcommand{\pc}{principal curvature}
\newcommand{\ct}{curvature tensor}
\newcommand{\cs}{conformal structure}
\newcommand{\Cs}{Christoffel symbol}
\newcommand{\cd}{complex dimension}
\newcommand{\Bd}{Beltrami differential}
\newcommand{\hBd}{harmonic Beltrami differential}
\newcommand{\sff}{second fundamental form}

\newcommand{\Hd}{Hopf differential}
\newcommand{\ts}{tangent space}
\newcommand{\tp}{tangent plane}
\newcommand{\tv}{tangent vector}
\newcommand{\tm}{3-manifold}
\newcommand{\ms}{moduli space}

\newcommand{\sbd}{surface bundle}

\newcommand{\ppl}[2]{\frac{\partial{#1}}{\partial{#2}}}

\begin{document}

\title{Curvatures on the Teichm\"{u}ller Curve}
\author{Ren Guo}
\address{School of Mathematics, University of Minnesota, Minneappolis, MN 55455, USA}
\email{guoxx170@math.umn.edu}

\author{Subhojoy Gupta}
\address{Department of Mathematics, Yale University, New Haven, CT 06520, USA}
\email{subhojoy.gupta@yale.edu}

\author{Zheng Huang}
\address{Department of Mathematics, 
The City University of New York, 
Staten Island, NY 10314, USA.}
\email{zheng.huang@csi.cuny.edu}

\date{January 16, 2010}

\subjclass[2000]{Primary 32G15, Secondary 53C43, 53C21}

\keywords{Teichm\"{u}ller space, Teichm\"{u}ller curve, sectional curvature, Weil-Petersson geodesic}

\begin{abstract}
The {\Tc} is the fiber space over {\TS} $T_g$ of closed {\RS}s, where the fiber over a point $(\Sigma,\sigma) \in T_g$ 
is the underlying surface $\Sigma$. We derive formulas for {\Sc}s on the {\Tc}. In particular, our method can be applied 
to investigate the geometry of the {\WPg} as a three-manifold, and the degeneration of the curvatures near the infinity of 
the augmented {\TS} along a {\WPg}, as well as the minimality of hyperbolic surfaces in this {\tm}.
\end{abstract}

\maketitle

\section {Introduction} 
{\TS} $T_{g}$ is the space of {\hym}s on $\Sigma$, modulo an equivalent relationship, where two {\cs}s $\sigma$ and 
$\rho$ are considered equivalent if there is a biholomorphic map between $(\Sigma,\sigma)$ and $(\Sigma,\rho)$, in 
the homotopy class of the identity. Here and throughout this paper, we always assume $\Sigma$ is a smooth, oriented, 
closed {\RS} of genus $g > 1$.

{\TS} is a complex manifold of {\cd} $3g - 3$ (\cite {Ahl61}),  and the co{\ts} at a base point $(\Sigma, \sigma)$ is identified 
with $Q(\sigma)$, the space of {\hqd}s on $(\Sigma, \sigma)$. Let $\{\phi_1,\cdots,\phi_{3g-3}\}$ be a basis of $Q(\sigma)$. 
The local coordinates of the {\TS} $T_g$ in the neighborhood of $(\Sigma, \sigma)$ are given by 
$(t^1,t^2,...,t^{3g-3}) \in \C^{3g-3}$. A generic {\hqd} $\phi dz^2$ on $(\Sigma, \sigma)$ is written as 
$\phi dz^2=\sum_{k=1}^{3g-3}t^k\phi_k dz^2$, where $z$ is the conformal coordinate on $(\Sigma,\sigma)$.

Throughout the paper, $\sigma$ and $\rho$ are {\cs}s, with conformal coordinates $z$ and $w$, respectively. We also 
denote $g_{\sigma}dzd\bar{z}$ and $g_{\rho}dwd\overline{w}$ as the {\hym}s on $(\Sigma,\sigma)$ and $(\Sigma, \rho)$, 
respectively. Corresponding to each $\phi dz^2$, there is a {\hym} $g_{\rho}dwd\overline{w}$ on $\Sigma$, from the work 
of Samson-Wolf \cite{Sam78, Wol89}: given a pair of points (two {\cs}s) $(\sigma, \rho)$ in {\TS}, there is a unique {\hmp} 
$w: (\Sigma, \sigma) \to (\Sigma, \rho)$ in the homotopy class of identity, and $\phi dz^2$ is the {\Hd} of the map $u$. This 
induces a homeomorphism $\phi: T_g \to Q(\sigma)$ which sends $\rho$ to $\phi dz^2$. Above correspondence is 
well-defined via the inverse of this homeomorphism.

The {\Tc} $\Tcal_g$ is the fiber space over the {\TS} $T_g$ of closed {\RS}s, where the fiber over a point 
$(\Sigma,\sigma) \in T_g$ is the underlying surface $\Sigma$. This is a manifold of real dimension $6g-4$. In this paper, 
we obtain curvature formulas for a general Riemannian metric on the {\Tc} $\Tcal_g$, and use this to study the geometry 
of a Riemannian {\tm} formed by a {\WPg}, particularly, the degeneration of its curvatures as the {\WPg} heads towards the 
boundary of the augmented {\TS}.

A point in $\Tcal_g$ is represented as $(\sigma,z_0)$, where $\sigma \in T_g$ and $z_0$ is a point on the marked surface 
$(\Sigma,\sigma)$. Let $\pi: T\Tcal_g \to TT_g$ be the fiberation and the kernel of the differential map 
$d\pi:T\Tcal_g \to TT_g$ defines a line bundle $\nu$ over the {\Tc} $\Tcal_g$. Wolpert \cite{Wol86} calculated the Chern 
form $c_{1}(v)$ of such a line bundle and showed that it is a negative differential $2$-form. He suggested to define a {\km} 
$G$ on the {\Tc} $\Tcal_g$ such that its {\kf} is $-c_{1}(v)$. The {\kp} of $G$ is given by $\log\|{\ppl {}{w}}\|$, where 
$\|\cdot\|$ is the length of $\cdot$ with respect to the {\hym} on a fiber $(\Sigma, \rho(w))$. Here $\rho(w)$ is the {\cs} 
$\rho$ with the conformal coordinate $w$. When restricted to each fiber $(\Sigma, \rho(w))$, the metric $G$ is 
${\frac{1}{2}}g_{\rho(w)}dwd\overline{w}$. When evaluated at $\sigma \in T_g$, the metric 
$G = \frac12\sigma dzd\overline{z}+ \sum_{k,\ell=1}^{3g-3}\delta_{k\ell}dt^k d\overline{t^{\ell}}$. Under this metric $G$ on 
the {\Tc} $\Tcal_g$, Jost calculated the holomorphic {\Sc} in the fiber direction:

\begin{1}\cite{Jos91a}
Under metric $G$, at the base point $(\sigma,z_0)$, the {\Sc} of the tangent plane 
expanded by $\frac{\partial}{\partial z}$ and $\sqrt{-1}\frac{\partial}{\partial z}$ is 
\begin{align}\label{J}
K(\frac{\partial}{\partial z}, \sqrt{-1}\frac{\partial}{\partial z}) = -1+\sum_{\ell=1}^{3g-3}|\mu_{\ell}|^2(z_0).
\end{align}
where $\{\mu_{\ell}(z)\}$ is a basis for the {\ts}, denoted by $B_{h}(\sigma)$, of {\TS} at $(\Sigma,\sigma)$, normalized 
according to $D(\mu_{\alpha}\overline{\mu}_{\beta})(z_0) =\delta_{\alpha\beta}$, for the operator 
$D = -2(\Delta_{g_{\sigma}}-2)^{-1}$ on the hyperbolic surface $(\Sigma, g_{\sigma}dzd\bar{z})$.
\end{1}

In the present work, using a general Riemannian metric on $\Tcal_g$, we find the same curvature formula holds in the fiber 
directions. Moreover, we determine the {\Sc}s of the directions spanned by one fiber direction and the other by a tangent 
vector in {\TS}.

Denote $z=x+\sqrt{-1}y$ and $t^{\ell}=x^{\ell}+\sqrt{-1}y^{\ell}.$ We consider a Riemannian metric $G'$ on the {\Tc} $\Tcal_g$:
\begin{equation}\label{metric}
G'= g_{\rho(w)}dwd\overline{w}+ \sum h_{\alpha\beta}(w)d\nu^\alpha d\nu^\beta
\end{equation}
where $\nu^\alpha, \nu^\beta\in \{x^1, y^1,...,x^{3g-3}, y^{3g-3}\}$. The metric $G'$ restricted on the fiber 
$(\Sigma,\rho)$ is identified with the {\hym} $g_{\rho(w)}dwd\overline{w}$.

\begin{theorem}\label{thm:curve}
With respect to the metric $G'$ on the {\Tc} $\Tcal_g$,  we assume functions $\{h_{\alpha\beta}\}$ satisfy the 
following: 
\begin{equation}\label{h1}
\sum h_{\alpha\beta}(z)d\nu^\alpha d\nu^\beta=2\sum_{\ell=1}^{3g-3}((dx^\ell) ^2+(dy^\ell) ^2), \forall z \in (\Sigma, \sigma),
\end{equation}
i.e., it is Euclidean when restricted on $(\Sigma, \sigma)$ in $T_g$. Then at the base point $(\sigma, z_0)$, the {\Sc} of 
the {\tp} expanded by $\frac{\partial}{\partial x}$ and $\frac{\partial}{\partial y}$ is 
\begin{align}\label{k-fiber}
K(\frac{\partial}{\partial x}, \frac{\partial}{\partial y}) = -1+\sum_{\ell=1}^{3g-3}|\mu_\ell|^2(z_0),
\end{align}
where $\mu_\ell=\frac{\overline{\phi_{\ell}}}{g_{\sigma}}$ and the basis $\{\phi_{\ell}dz^2\}$ on $Q(\sigma)$ is chosen such 
that \eqref{h1} holds.
\end{theorem} 

If we require further that functions $\{h_{\alpha\beta}\}$ satisfy the following: 
\begin{equation}\label{h2}
\sum h_{\alpha\beta}(w)d\nu^\alpha d\nu^\beta=2\sum_{\ell=1}^{3g-3}((dx^\ell)^2+(dy^\ell)^2), \forall \rho(w) \in T_g,
\end{equation}
i.e., a global Euclidean structure rather than a local one as in \eqref{h1}, then we have: 

\begin{theorem}\label{thm:ft}
With respect to the metric $G'$ on the {\Tc} $\Tcal_g$,  we choose the basis  $\{\phi_{\ell}dz^2\}$ on $Q(\sigma)$ 
such that \eqref{h2} holds for functions $\{h_{\alpha\beta}\}$ in $G'$. Then at the base point $(\sigma, z_0)$, the 
{\Sc}s of the {\tp}s expanded by $\frac{\partial}{\partial x}$ and ${\frac{\partial}{\partial x^\ell}}$, 
$\ell = 1, 2, \cdots, 3g-3$, are:
\begin{align}\label{fib-tan1}
K(\frac{\partial}{\partial x}, {\frac{\partial}{\partial x^\ell}}) = -{\frac{1}{2}}D(|\mu_{\ell}|^2)(z_0), 
\end{align}
where $\mu_{\ell} =\frac{\bar{\phi}_{\ell}}{g_{\sigma}}$.
Similarly,
\begin{align}\label{fib-tan2}
K(\frac{\partial}{\partial y}, {\frac{\partial}{\partial x^\ell}}) = K(\frac{\partial}{\partial x}, {\frac{\partial}{\partial y^\ell}}) 
= K(\frac{\partial}{\partial y}, {\frac{\partial}{\partial y^\ell}}) = -{\frac{1}{2}}D(|\mu_{\ell}|^2)(z_0).
\end{align}
They are all non-positive.
\end{theorem} 

One of the motivations of this paper is to study the {\WPg}s in {\TS} from its intrinsic geometry. As {\cs}s of a 
topological surface travel along a curve in {\TS}, they form a three-space which is homeomorphic to 
$\Sigma \times \R$. We would like to set up as follows to study the shape of a {\WPg}: we take one tangent vector 
$\mu_0 {\frac{d\bar{z}}{dz}}$ at the point $\sigma \in T_g$. Consider the {\WPg} $\gamma = \gamma(s)$ in {\TS} 
$T_g$ through the point $\sigma$ and in the direction $\mu_0 {\frac{d\bar{z}}{dz}}$. We consider the germ 
$N_{\sigma}$ at $\sigma$, and a natural local metric denoted by $H$, near $t=0$, takes a very simple form:  
\begin{equation}\label{H}
H = g_{\rho(w)}dwd\overline{w} + dt^2,
\end{equation}
where $g_{\rho(w)}dwd\overline{w}$ is the pull-back metric will be made clear in $(3.1)$.

\begin{theorem}\label{thm:g}
At any point $z_0 \in (\Sigma, \sigma)$, we determine the {\Sc}s of the {\tp}s spanned by two of the three vectors 
$\frac{\partial}{\partial x}$, $\frac{\partial}{\partial y}$, and $\frac{\partial}{\partial t}$, with respect to the metric $H$ 
on a {\WPg} $\gamma$ through $\sigma$ in the direction of $\mu_0 {\frac{d\bar{z}}{dz}} \in B_h(\sigma)$, where there is 
a $\phi_0dz^2 \in Q(\sigma)$ such that $\mu_0 {\frac{d\bar{z}}{dz}}=\frac{\bar{\phi}_0d\bar{z}^2}{g_\sigma |dz|^2}$:
\begin{enumerate}
\item
$K(\frac{\partial}{\partial x}, \frac{\partial}{\partial y})(\sigma,z_0) = -1+|\mu_0(z_0)|^2$;
\item
$K(\frac{\partial}{\partial x}, \frac{\partial}{\partial t})(\sigma,z_0) = 
K(\frac{\partial}{\partial y}, \frac{\partial}{\partial t}) (\sigma,z_0)= -D(|\mu_0|^2)(z_0)$.
\end{enumerate}
\end{theorem}

Moreover, we define a Riemannian structure on the surface bundle $N = \bigcup_{\sigma \in \gamma}N_\sigma$ over 
the {\WPg} $\gamma$, and find
\begin{theorem}
Let $\gamma$ be a {\WPg} in {\TS}. Then there exists a choice of a metric in each conformal class along $\gamma$, such 
that the surface bundle $N$ over $\gamma$ acquires a Riemannian metric whose germ at each fiber is $N_\sigma$ as in 
\eqref{H}. In particular,
\begin{enumerate}
\item
each fiber over $\sigma$ is a hyperbolic surface;
\item
each fiber over $\sigma$ is also a minimal surface;
\item 
the sectional curvatures of the metric are given in the Theorem 1.3. 
\end{enumerate}
\end{theorem}
Generally, it is difficult to carry out calculations on the {\WP} metric on {\TS}. Part of the reason is that the metric itself, 
as well as the curvature tensor formula, are given as an integral (or a sum of integrals) over the underlying {\cs} 
$(\Sigma, \sigma)$, rather than in terms of any global coordinates of {\TS}. In this work, the technical tool we intensively 
rely on is the {\Tt} of {\hmp}s, begun with the pioneer work of Eells-Sampson (\cite {ES64}). This theory has many 
applications in many different areas in geometry and analysis. In the case of compact hyperbolic surfaces, the analytical 
theory does not involve any regularity issue: the degree one {\hmp} between hyperbolic surfaces is a diffeomorphism 
(\cite {SY78}). It has become an important computational tool in {\Tt}, as well as hyperbolic geometry (see, for example, 
\cite {Wol89}, \cite{Jos91a, Jos91b, JY09}, \cite {Min92a, Min92b},  \cite{Hua05, Hua07}).   
\subsection*{Plan of the paper}
We start with a brief collection of preliminary facts in \S 2, where we introduce {\hmp}s between 
hyperbolic surface in \S 2.1, {\TS} and the {\Tc} in \S 2.2, and the {\WPm} in \S 2.3. Theorems \ref{thm:curve} and 
\ref{thm:ft} are proved in \S 3, where we obtain formulas for the curvatures of the {\Tc} for the Riemannian metric $G'$ 
(\eqref{metric}). We will focus on the geometry of the {\WPg} in the last section \S 4, prove the Theorems \ref{thm:g} 
and 1.4.
\subsection*{Acknowledgment}
We wish to express our gratitude to J\"{u}rgen Jost, Yair Minsky for their generous help and interest, and especially 
Michael Wolf and Scott Wolpert for their invaluable suggestions. The research of Huang is partially supported by a 
PSC-CUNY research grant.
\section{Background}
\subsection{Harmonic maps between surfaces}
We review our basic computational scheme, which is based on {\hmp}s between compact hyperbolic surfaces.

Let $w:(\Sigma, g_{\sigma}|dz|^2) \rightarrow (\Sigma, g_{\rho}|dw|^2)$ be a Lipschitz map, where $g_{\sigma}|dz|^2$ 
and $g_{\rho}|dw|^2$ are {\hym}s on the surface $\Sigma$, associated to the {\cs}s $\sigma$ and $\rho$, respectively. 
And $z$ and $w$ are conformal coordinates on $(\Sigma, \sigma)$ and $(\Sigma,\rho)$. We (following \cite {Sam78}) 
define important density functions: the {\it holomorphic energy density}
\begin{equation}
{\mathcal{H}}(z) = {\frac{g_{\rho}}{g_{\sigma}}}|w_z|^2, 
\end{equation}
and the {\it anti-holomorphic energy density}
\begin{equation}
\Lcal(z) =  {\frac{g_{\rho}}{g_{\sigma}}}|w_{\bar{z}}|^2.
\end{equation}

The energy density function of the map $w$ is now simply 
\begin{equation}\label{e}
e(w(z))= \Hcal(z) + \Lcal(z),
\end{equation}
and the {\it total energy} and the {\it Jacobian determinant} of the map are given by
\begin{equation}\label{E}
E(w,\sigma,\rho) = \int_{\Sigma}e(z)g_{\sigma}|dz|^2 = \int_{\Sigma}e(z)dA
\end{equation}
and
\begin{equation}
\Jcal(z) = \Hcal(z) - \Lcal(z),
\end{equation}
respectively. Here $dA$ in \eqref{e} is the area element for $(\Sigma, \sigma)$.

The map $w$ is {\it harmonic} if it is a critical point of this total energy functional \eqref{E}. The $(2,0)$ part of the 
pullback $w^{*}\rho$ is particularly important, and it is called {\it {Hopf differential}} of $w$:
\begin{equation*}
\phi(z)dz^2 = (w^{*}\rho)^{(2,0)} = g_{\rho}w_z {\bar{w}}_zdz^2.
\end{equation*}

It is well-known that there is a unique {\hmp} $w:(\Sigma, \sigma) \rightarrow (\Sigma, \rho)$ in each homotopy class, 
and $w$ is a diffeomorphism with positive Jacobian determinant. 

\subsection{{\TS} and the {\Tc}}
For a closed surface $\Sigma$ of genus $g > 1$, it is a consequence of the Uniformization Theorem that the notion of 
{\cs}s, complex structures, and {\hym}s on $\Sigma$ are equivalent. {\TS} $T_g$ is the space of {\cs}s on $\Sigma$, 
modulo the group of orientation preserving diffeomorphisms isotopic to the identity.

For a {\cs} $\sigma$ on $\Sigma$, it represents a point in $T_g$, and we denote by $z$ its conformal coordinate. We 
routinely use $(\Sigma, \sigma)$ to indicate the marking by a {\cs}. The co{\ts} of $T_g$ at $\sigma$ is identified as 
$Q(\sigma) = \{\phi(z)dz^2: \overline{\partial} \phi =0\}$, the space of {\hqd}s on $(\Sigma, \sigma)$. The {\Hd} of a 
{\hmp} is holomorphic, hence belongs to $Q(\sigma)$, and the map $w$ is conformal if and only if its {\Hd} is $0$. This 
is essentially the entrance of {\hmp} theory to {\Tt}. Note that $Q(\sigma)$ is a Banach space of complex dimension 
$3g-3$, while the space of {\Bd}s is of infinite dimension.

We denote $g_{\sigma}|dz|^2$ the {\hym} on $(\Sigma, \sigma)$, and its Laplacian is 
\begin{equation*}
\Delta_{\sigma} = {\frac{4}{g_{\sigma}}}{\frac{\partial^2}{\partial z \partial \bar{z}}},
\end{equation*}
with nonpositive eigenvalues. We also define an operator $D = -2(\Delta_{\sigma} -2)^{-1}$. It is $L^{2}$-self-adjoint 
with respect to $(\Sigma, g_{\sigma}|dz|^2)$. This operator plays an essential role in understanding the {\WP} 
geometry of {\TS}.

The {\ts} of $T_g$ at $\sigma$ can be identified with the space of {\hBd}s $B_h(\sigma)$. A {\Bd} 
$\mu(z){\frac{d\bar{z}}{dz}}$ on $(\Sigma, \sigma)$ is harmonic if 
$\mu(z){\frac{d\bar{z}}{dz}} = {\frac{\bar{\phi}d\bar{z}^2}{g_{\sigma}dzd\bar{z}}}$ for some $\phi dz^2 \in Q(\sigma)$.

The {\Tc} $\Tcal_g$ is a fiber bundle over {\TS} $T_g$, and the fiber over $\sigma \in T_g$ is the marked surface 
$(\Sigma, \sigma)$. As a manifold of real dimension $6g -4$, every point in $\Tcal_g$ can be represented as 
$(\sigma, z_0)$, where $\sigma \in T_g$ and $z_0 \in (\Sigma, \sigma)$.

 The first Chern class $c_1(v) = \frac{\sqrt{-1}}{2\pi} \Theta$ of the line bundle $\nu$ over $\Tcal_g$ is computed 
 by Wolpert (\cite{Wol86}), where the curvature 2-form $\Theta$ is found to satisfy 
 $\Theta(\ppl{}{z}, \ppl{}{\bar{z}}) = \frac{-2}{(z-\bar{z})^2}$, and $\Theta(\ppl{}{\bar{z}}, \tau_{\mu}) = 0$,  and 
 $\Theta(\bar{\tau}_{\nu}, \tau_{\mu}) = D(\mu\bar{\nu})$. It is negative, therefore one can define a {\km} $-c_1(v)$ 
 on $\Tcal_g$. 
\subsection{The {\WPm} and {\tm}s} 
The {\WP} co-metric is defined on the cotangent space $Q(\sigma)$ by the natural $L^2$-norm:
\begin{eqnarray}
||\phi||_{WP}^2 = \int_{\Sigma} \frac {|\phi|^2}{g_{\sigma}^2}dA_{\sigma}, \forall \phi \in Q(\sigma),
\end{eqnarray}
where $dA_{\sigma}$ is the hyperbolic area element of $(\Sigma,\sigma)$. By duality, we obtain the {\WPm} on $T_g$.

The {\WPg}s are intimately related to the geometry of three manifolds. Given $X, Y \in {T}_g(\Sigma)$, they uniquely 
determine a {\qf} hyperbolic {\tm}, $Q(X,Y)$, with $X$ and $Y$ as conformal boundaries, by the means of 
Bers' simultaneous uniformization (\cite{Ber72}). Brock (\cite{Bro03}) showed that the hyperbolic volume of 
the convex core of $Q(X,Y)$ is quasi-isometric to the length of the {\WPg} joining $X$ and $Y$ in {\TS}. We 
obtained a {\WP} potential from varying {\qf} manifolds in {\qf} space near the Fuchsian locus (\cite{GHW09}).

Many mysterious properties of the {\WP} geometry of {\TS} are largely due to the incompleteness of the metric 
(\cite {Chu76, Wol75}). When a {\WP} geodesic can not be extended, a short simple closed curve on the 
surface is pinched to a single point (\cite{Mas76}), while curvatures on the surface remain hyperbolic. A basic 
property is that the {\WPm} is geodescially convex (\cite {Wol87}), hence any two points can be joined by a 
unique {\WP} geodesic. The {\WP} geometry of {\TS} is quite satisfying: it is a space of negative curvature 
(\cite {Tro86}, \cite {Wol86}). However, there is neither negative upper bound (\cite {Hua05}) nor lower bound 
(\cite {Hua07}, \cite {Sch86}) of the sectional curvatures. We refer more detailed discussions on the {\WP} 
geometry of {\TS} to articles (\cite {Wol03, Wol06, Wol09}).

We want to understand the infinitesimal geometry of the {\WPg}. To this end, we consider a simple {\WPg} $\gamma$ 
determined by the point $\sigma \in T_g$ and the direction $\mu_{\ell}$, where $\mu_{\ell} \in B_h(\sigma)$ is a 
tangent vector at $\sigma$. We denote $N_\sigma$ the germ of a hyperbolic surface associated to the {\cs} at the 
point $\sigma \in \gamma$. The set of these germs over the {\WPg} $\gamma$ form a three-dimensional space. 

\section{Curvature Formulas of the {\Tc}}
We prove theorems 1.1 and 1.2 in this section. Curvatures in the fiber directions are obtained in \S 3.1, and fiber-tangential 
directions follow in \S 3.2.
\subsection{Fiber directions}
In this subsection, we determine the {\Sc} of the {\Tc} $\Tcal_g$ in the fiber directions. The metric $G'$ is given as in 
\eqref{metric}:
\begin{equation*}
G'= g_{\rho(w)}dwd\overline{w}+ \sum h_{\alpha\beta}(w)d\nu^\alpha d\nu^\beta,
\end{equation*}
where the functions $\{h_{\alpha\beta}\}$ satisfy \eqref{h1}: 
\begin{equation*}
\sum h_{\alpha\beta}(z)d\nu^\alpha d\nu^\beta=2\sum_{\ell=1}^{3g-3}((dx^\ell)^2+(dy^\ell)^2), \forall z \in (\Sigma, \sigma).
\end{equation*}
In other words, we only require it to be Euclidean at $w=z$ in $T_g$.

We now proceed to calculate the {\Sc} in the fiber directions, spanned by vectors $\frac{\partial}{\partial x}$ and 
$\frac{\partial}{\partial y}$, where $z =x+\sqrt{-1}y$. 

\begin{proof} [Proof of Theorem \ref{thm:curve}] 

Fixing $\sigma \in T_g$, for any {\hym} $g_{\rho}$ on $\Sigma$, denote the unique 
{\hmp} $w: (\Sigma,\sigma) \rightarrow (\Sigma,g_{\rho}dwd\bar{w})$ in the homotopy class of the identity, one 
obtains the {\Hd} $\phi(z)dz^2$ of $w$, which is holomorphic. Therefore we have a map $\phi$ from {\TS} to 
$Q(\sigma)$. This map is a homeomorphism (\cite {Sam78, Wol89}). Therefore, via the inverse map, $Q(\sigma)$ 
provides global coordinates for $T_g$.

The space of {\hqd}s $Q(\sigma)$ is a Banach space, so for $\phi_{0}dz^{2} \in Q(\sigma)\backslash \{0\}$, and 
$\phi(t) = t\phi_{0}$ is a ray in $Q(\sigma)$. We denote the 
{\hym}s $g_{\rho(t)}|dw(t)|^{2}$ as the points in $T_g$ determined by the ray $\phi(t)$ in $Q(\sigma)$, via the 
Sampson-Wolf Theorem. Here $w(t)$ is the family of {\hmp}s, in the homotopy class of the identity, whose associated 
{\Hd}s are given by $\phi(t) =  \rho(w(t)) w_z(t) {\bar{w}}_z(t)dz^2$.

Clearly, at $t = 0$, $\rho(0) = \sigma$.

For this family of {\hmp}s $w(t): (\Sigma, g_{\sigma}|dz|^2) \rightarrow (\Sigma, g_{\rho(t)}|dw(t)|^{2})$, we pull back 
the metric to find
\begin{eqnarray}\label{pb}
w^{*} g_{\rho(t)}|dw(t)|^{2} &=& 2Re(g_{\rho(w(t))}w_z(t){\bar{w}}_z(t)dz^2) + 
g_{\rho(t)}(|w_{z}(t)|^{2}+ |{\bar{w}}_{z}(t)|^{2}) |dz|^{2} \nonumber \\
&=& \phi(t) dz^{2} + g_{\sigma}e(t)|dz|^{2} + {\bar{\phi}}(t)d{\bar{z}}^{2}, 
\end{eqnarray}
where $e=e(t)$ is the energy density of $w(t)$, as in \eqref{e}.

For $\phi(t)dz^{2}$, we have $\mu(t)\frac{d\overline{z}}{d z} = {\frac{\bar{\phi}(t)d\bar{z}^{2}}{g_\sigma dzd{\bar{z}}}} \in B_h(\sigma)$, a family 
of {\hBd}s. They represent {\tv}s at $\sigma \in T_g$.

Let $\{\phi_1,\cdots,\phi_{3g-3}\}$ be a basis of $Q(\sigma)$ such that \eqref{h1} holds for the functions 
$\{h_{\alpha\beta}\}$ in the definition \eqref{metric} of the metric $G'$ of $\Tcal_g$.

We write $t =(t^1, \cdots, t^{3g-3})$ such that $\phi(t) dz^2=\sum_{\ell=1}^{3g-3}t^{\ell}\phi_{\ell}dz^2$. And we use $|_0$ to 
indicate evaluation at $t^{\alpha} = 0$ for any $\alpha =1, \cdots, 3g-3$.

The variations of the energy density $e(t)$ are calculated by Wolf (\cite{Wol89}) as follows:
\begin{align*}
 e(t)|_0&=1,\\
 \frac{\partial e(t)}{\partial t^\alpha}|_0 &=\frac{\partial e(t)}{\partial \overline{t^{\alpha}}} |_0 =0,\\
 \frac{\partial^2 e(t)}{\partial t^\alpha \partial t^\beta}|_0 &
 =\frac{\partial^2 e(t)}{\partial \overline{t^{\alpha}} \partial \overline{t^{\beta}}}|_0 =0,\\
 \frac{\partial^2e(t)}{\partial t^{\alpha} \partial \overline{t^{\beta}}}|_0 &
 =(D+1)\frac{\phi_\alpha \overline{\phi}_{\beta}}{g_{\sigma}^2}.
\end{align*}
Using real coordinates, we can rewrite the above results as:
\begin{align*}
 e(t)|_0&=1,\\
 \frac{\partial e(t)}{\partial x^\alpha}|_0 &=\frac{\partial e(t)}{\partial y^\alpha } |_0 =0,\\
 \frac{\partial^2 e(t)}{\partial x^\alpha \partial x^\beta}|_0 &=\frac{\partial^2 e(t)}{\partial y^\alpha \partial y^\beta}|_0 
 =(D+1)\frac{2Re(\phi_\alpha \overline{\phi}_{\beta}) }{g_{\sigma}^2},
\end{align*}

Using the pullback \eqref{pb}, the metric $G'$ on $\Tcal_g$ is written as 
\begin{align}\label{pb-r}
G'&=\phi(t)dz^2+g_{\sigma}e(t)dzd\overline{z} + \overline{\phi}(t) d\overline{z}^2+ 
\sum h_{\alpha\beta}(w)d\nu^\alpha d\nu^\beta \nonumber\\
    &=(g_{\sigma}e+ 2Re \phi)dx^2-4(Im\phi)dxdy+(g_{\sigma}e- 2Re \phi)dy^2 + \sum h_{\alpha\beta}(w)d\nu^\alpha d\nu^\beta.
\end{align}

To simply our notation, we denote $R_{1221} = R_{xyyx}$, and utilize Einstein notation to compute this curvature tensor 
as follows:
\begin{align}\label{r}
R_{1221}|_0
&\ =g_{11}(\Gamma_{22,1}^1-\Gamma_{12,1}^1+
\Gamma_{22}^\beta\Gamma_{1\beta}^1-\Gamma_{12}^\beta\Gamma_{2\beta}^1)|_0\nonumber \\
&\ =g_{11}(\Gamma_{22,1}^1-\Gamma_{12,1}^1\nonumber\\
&\ \ \ \ \ \ \ \ \ \ +\Gamma_{22}^1\Gamma_{11}^1+\Gamma_{22}^2\Gamma_{12}^1
+\sum_{\ell=1}^{3g-3}(\Gamma_{22}^{x^\ell}\Gamma_{1x^\ell}^1+\Gamma_{22}^{y^\ell}\Gamma_{1y^\ell}^1)\nonumber\\
&\ \ \ \ \ \ \ \ \ \ -\Gamma_{12}^1\Gamma_{21}^1-\Gamma_{12}^2\Gamma_{22}^1
-\sum_{\ell=1}^{3g-3}(\Gamma_{12}^{x^\ell}\Gamma_{2x^\ell}^1+\Gamma_{12}^{y^\ell}\Gamma_{2y^\ell}^1))|_0.
\end{align}
Here we recall from \eqref{h1} that, at $t=0$, the functions $h_{\alpha\beta}$ satisfy the condition that 
$\sum h_{\alpha\beta}(z)d\nu^\alpha d\nu^\beta=2\sum_{\ell=1}^{3g-3}((dx^\ell)^2+(dy^\ell)^2)$. 

We calculate the values of Christoffel symbols evaluated at $t^\alpha=0$ for any $\alpha$:
$$
\left\{
\begin{matrix}
\Gamma_{11}^{1}=\displaystyle \frac{(g_{\sigma})_1}{2g_{\sigma}}, 
& \Gamma_{11}^{2}= \displaystyle -\frac{(g_{\sigma})_2}{2g_\sigma},
& \Gamma_{11}^{x^\ell}= \displaystyle -\frac{Re\phi_\ell}2,  
& \Gamma_{11}^{y^\ell}= \displaystyle \frac{Im\phi_\ell}2,\\
\\
\Gamma_{12}^{1}=\displaystyle \frac{(g_{\sigma})_2}{2\sigma}, 
& \Gamma_{12}^{2} =\displaystyle \frac{(g_{\sigma})_1}{2\sigma},
& \Gamma_{12}^{x^\ell} = \displaystyle \frac{Im\phi_\ell}2,
& \Gamma_{12}^{y^\ell}= \displaystyle \frac{Re\phi_\ell}2, \\
\\
\Gamma_{22}^{1}=\displaystyle -\frac{(g_{\sigma})_1}{2g_\sigma}, 
& \Gamma_{22}^{2}=\displaystyle \frac{(g_{\sigma})_2}{2g_\sigma}, 
& \Gamma_{22}^{x^\ell}=\displaystyle \frac{Re\phi_\ell}2, 
& \Gamma_{22}^{y^\ell}=\displaystyle -\frac{Im\phi_\ell}2.
\end{matrix}
\right.
$$
Some mixed terms are calculated as follows:
$$
\left\{
\begin{matrix}
\Gamma_{1x^\ell}^{1}=\displaystyle \frac{Re\phi_\ell}{g_{\sigma}}, 
& \Gamma_{1x^\ell}^{2}= \displaystyle -\frac{Im\phi_\ell}{g_{\sigma}},
& \Gamma_{1x^\ell}^{x^\ell}= \displaystyle  0,
& \Gamma_{1x^\ell}^{y^\ell}= \displaystyle 0, \\
\\
\Gamma_{1y^\ell}^{1}=\displaystyle -\frac{Im\phi_\ell}{g_{\sigma}}, 
& \Gamma_{1y^\ell}^{2} =\displaystyle -\frac{Re\phi_\ell}{g_{\sigma}},
& \Gamma_{1y^\ell}^{x^\ell} =\displaystyle 0, 
& \Gamma_{1y^\ell}^{y^\ell} =\displaystyle 0,\\
\\
\Gamma_{2x^\ell}^{1}=\displaystyle -\frac{Im\phi_\ell}{g_{\sigma}}, 
& \Gamma_{2x^\ell}^{2}= \displaystyle -\frac{Re\phi_\ell}{g_{\sigma}},
& \Gamma_{2x^\ell}^{x^\ell}= \displaystyle  0,
& \Gamma_{2x^\ell}^{y^\ell}= \displaystyle 0 ,\\
\\
\Gamma_{2y^\ell}^{1}=\displaystyle -\frac{Re\phi_\ell}{g_{\sigma}},
& \Gamma_{2y^\ell}^{2}=\displaystyle \frac{Im\phi_\ell}{g_{\sigma}},
& \Gamma_{2y^\ell}^{x^\ell}=\displaystyle 0,
& \Gamma_{2y^\ell}^{y^\ell}=\displaystyle 0,
\end{matrix}
\right.
$$
and
$$
\left\{
\begin{matrix}
\Gamma_{x^\ell x^\ell}^{1}=\displaystyle 0, 
& \Gamma_{x^\ell x^\ell}^{2}= \displaystyle 0,\\
\\
\Gamma_{x^\ell y^\ell}^{1}=\displaystyle 0 , 
& \Gamma_{x^\ell y^\ell}^{2} = \displaystyle 0,\\
\\
\Gamma_{y^\ell y^\ell}^{1}=\displaystyle 0 , 
& \Gamma_{y^\ell y^\ell}^{2}= \displaystyle 0.\\
\end{matrix}
\right.
$$
Using that the curvature of the {\hym} $g_{\sigma}(dx^2+dy^2)$ is $-1$, and $g_{11}|_0 = g_{\sigma}$ from \eqref{pb-r}, we have:
\begin{align*}
R_{1221}|_0
&\ =g_{\sigma}(-(\frac{(g_\sigma)_1}{2g_{\sigma}})_1-(\frac{(g_\sigma)_2}{2g_{\sigma}})_2\\
&\ \ \ \ \ \ \ \ \ \ -\frac{((g_\sigma)_1)^2}{4g_{\sigma}^2}+\frac{((g_\sigma)_2)^2}{4g_{\sigma}^2}
+\frac12\sum_{\ell=1}^{3g-3}(\frac{(Re\phi_\ell)^2}{g_\sigma}+\frac{(Im\phi_\ell)^2}{g_{\sigma}})\\
&\ \ \ \ \ \ \ \ \ \ -\frac{((g_\sigma)_2)^2}{4g_\sigma^2}+\frac{((g_\sigma)_1)^2}{4g_\sigma^2}
+\frac12\sum_{\ell=1}^{3g-3}(\frac{(Re\phi_\ell)^2}{g_\sigma}+\frac{(Im\phi_\ell)^2}{g_\sigma}))\\
&=g_{\sigma}(-(\frac{(g_\sigma)_1}{2g_\sigma})_1-(\frac{(g_\sigma)_2}{2g_\sigma})_2+
\frac{1}{g_\sigma}\sum_{\ell=1}^{3g-3}|\phi_\ell|^2 )\\
&\ =-g_{\sigma}^2+\sum_{\ell=1}^{3g-3}|\phi_\ell|^2.
\end{align*}
Therefore the curvature in directions $\frac{\partial}{\partial x}$ and $\frac{\partial}{\partial y}$, at $t=0$, is 
\begin{eqnarray*}
K(\frac{\partial}{\partial x}, \frac{\partial}{\partial y}) = {\frac{R_{1221}}{g_\sigma^2}} &=&
-1+\sum_{\ell=1}^{3g-3}\frac{|\phi_{\ell}|^2}{g_{\sigma}^2}(z_0) \\
&=& -1+ \sum_{\ell=1}^{3g-3}|\mu_\ell|^{2}(z_0).
\end{eqnarray*}
\end{proof} 
The normalization for the metric $G = -c_{1}(v)$ on $\Tcal_g$ provides the following estimates of the curvatures 
in the fiber directions, in particular, it is impossible for the curvatures to be non-positive everywhere in $\Tcal_g$:

\begin{cor}
With respect to the metric $G = -c_{1}(v)$ on $\Tcal_g$, we have
\begin{enumerate}
\item
\cite{Jos91a}: $Sup_{(\sigma,z_0) \in \Tcal_g}K(\frac{\partial}{\partial z}, \sqrt{-1}\frac{\partial}{\partial z}) > 0$;
\item
$Sup_{(\sigma,z_0) \in \Tcal_g}K(\frac{\partial}{\partial z}, \sqrt{-1}\frac{\partial}{\partial z}) \le 9g-10$.
\end{enumerate}
\end{cor}

\begin{proof}
\vskip 0.1in
\noindent
\begin{enumerate}
\item
This is proved in (\cite{Jos91a}). The idea is that the explicit {\kp} of the metric $G$ forces $D(|\mu_\ell|^2)(z_0) =1$ 
for all $\ell = 1, \cdots, 3g-3$, then the lower bound $0$ is a consequence of the {\maxp}.
\item
The upper bound is a consequence of the following point-wise estimate from next Lemma. 
\end{enumerate}
\end{proof}
\begin{lem}\label{thm:w}\cite{Wol08}
For any $\mu(z){\frac{d\bar{z}}{dz}} \in B_h(\sigma)$, and $\forall z \in (\Sigma, \sigma)$, we have the following:
\begin{equation*}
3D(|\mu|^2)(z) \ge |\mu|^2(z).
\end{equation*}
\end{lem}
\subsection{Fiber-tangential directions} 
To calculate the {\Sc}s spanned by one fiber direction and one tangential direction, we have to complete the mixed 
terms in Christoffel symbols which were not required in the last subsection. For this reason, we require functions 
$\{h_{\alpha\beta}\}$ satisfy \eqref{h2}: 

$$\sum h_{\alpha\beta}(w)d\nu^\alpha d\nu^\beta=2\sum_{\ell=1}^{3g-3}((dx^\ell)^2+(dy^\ell)^2), 
\forall \rho(w) \in T_g.$$

Therefore the metric $G'$ takes the form:
 \begin{equation}\label{metric2}
G'= g_{\rho(w)}dwd\overline{w}+ 2\sum_{\ell=1}^{3g-3}(dx_\ell^2+dy_\ell^2).
\end{equation}
And we are now going to prove the Theorem \ref{thm:ft}.

\begin{proof} [Proof of Theorem \ref{thm:ft}]

We will continue to use notation in the proof of Theorem \ref{thm:curve}. We will only calculate the curvature tensor 
$R_{1x^{\ell}x^{\ell}1} = R_{xx^{\ell}x^{\ell}x}$ at $t=0$, which is given by:
\begin{equation}\label{r2}
R_{1x^{\ell}x^{\ell}1}|_0 = g_{\sigma}(\Gamma_{x^{\ell}x^{\ell},1}^1-\Gamma_{1x^{\ell},x^\ell}^1+
\Gamma_{x^{\ell}x^{\ell}}^\beta\Gamma_{1\beta}^1-\Gamma_{1x^{\ell}}^\beta\Gamma_{x^{\ell}\beta}^1)|_0.
\end{equation}
We have 
\begin{align*}
& \Gamma_{1x^{\ell}}^1=\frac12\frac{g_{\sigma}e- 2 Re \phi}{g_{\sigma}^2 e^2-4 |\phi|^2}
                              (g_{\sigma} \frac{\partial e}{\partial x^{\ell}}+ 2 Re \phi_{\ell})+
                \frac12\frac{2Im\phi            }{g_{\sigma}^2 e^2-4|\phi|^2}
                              (-2Im\phi_{\ell}),\\
& \Gamma_{1y^{\ell}}^1=\frac12\frac{g_{\sigma}e- 2 Re \phi}{g_{\sigma}^2 e^2-4 |\phi|^2}
                              (g_{\sigma}\frac{\partial e}{\partial y^{\ell}}- 2 Im \phi_{\ell})+
                \frac12\frac{2Im\phi            }{g_{\sigma}^2 e^2-4|\phi|^2}
                              (2Re\phi_{\ell}),\\
&\Gamma_{2x^{\ell}}^2=\frac12\frac{2Im\phi            }{g_{\sigma}^2 e^2-4|\phi|^2}
                              (-2Im\phi_{\ell})+
                \frac12\frac{g_{\sigma}e+ 2 Re \phi}{g_{\sigma}^2 e^2-4|\phi|^2}
                              (g_{\sigma}\frac{\partial e}{\partial x^{\ell}}- 2 Re \phi_{\ell}),\\
&\Gamma_{2y^{\ell}}^2=\frac12\frac{2Im\phi            }{g_{\sigma}^2 e^2-4|\phi|^2}
                              (2Re\phi_{\ell})+
                \frac12\frac{g_{\sigma}e+ 2 Re \phi}{g_{\sigma}^2 e^2-4|\phi|^2}
                              (\sigma \frac{\partial e}{\partial y^{\ell}}+ 2 Im \phi_{\ell}).
\end{align*}
And their derivatives at $t=0$ are as follows:
\begin{align}
& \Gamma_{1x^{\ell},x^{\ell}}^1|_0=\Gamma_{2x^{\ell},x^{\ell}}^1|_0=-2\frac{|\phi_{\ell}|^2}{g_\sigma^2}+
                       \frac12 \frac{\partial^2 e(w)}{\partial x^{\ell} \partial x^{\ell}}
                     =(D-1)\frac{|\phi_{\ell}|^2}{g_\sigma^2},\\
& \Gamma_{1y^{\ell},y^{\ell}}^1|_0=\Gamma_{2y^{\ell},y^{\ell}}^1|_0=-2\frac{|\phi_{\ell}|^2}{g_\sigma^2}+
                       \frac12 \frac{\partial^2 e(w)}{\partial y^{\ell} \partial y^{\ell}}
                     =(D-1)\frac{|\phi_{\ell}|^2}{g_\sigma^2}. 
\end{align}
Moreover, because of condition \eqref{metric2}, we can complete the table of Christoffel symbols by:
$$
\left\{
\begin{matrix}
& \Gamma_{x^{\ell}x^{\ell}}^{x^{\ell}}= \displaystyle 0 ,
& \Gamma_{x^{\ell}x^{\ell}}^{y^{\ell}}= \displaystyle 0 , \\
\\
& \Gamma_{x^{\ell}y^{\ell}}^{x^{\ell}} =\displaystyle 0,
& \Gamma_{x^{\ell}y^{\ell}}^{y^{\ell}} =\displaystyle 0,\\
\\
& \Gamma_{y^{\ell}y^{\ell}}^{x^{\ell}}= \displaystyle 0 ,
& \Gamma_{y^{\ell}y^{\ell}}^{y^{\ell}}= \displaystyle 0 .\\
\end{matrix}
\right.
$$
Substituting these terms into \eqref{r2}, we find 
\begin{align*}
R_{1x^{\ell}x^{\ell}1}|_0 &= g_{\sigma}(\Gamma_{x^{\ell}x^{\ell},1}^1-\Gamma_{1x^{\ell},x^\ell}^1+
\Gamma_{x^{\ell}x^{\ell}}^\beta\Gamma_{1\beta}^1-\Gamma_{1x^{\ell}}^\beta\Gamma_{x^{\ell}\beta}^1)|_0\\
&= g_{\sigma}(0-(D-1)\frac{|\phi_{\ell}|^2}{g_\sigma^2}+0-\frac{(Re\phi_{\ell})^2}{g_\sigma^2}-\frac{(Im\phi_{\ell})^2}{g_\sigma^2})\\
&= -g_\sigma D(\frac{|\phi_{\ell}|^2}{g_\sigma^2})(z_0).
\end{align*}
Then the curvature 
\begin{equation*}
K(\frac{\partial}{\partial x}, {\frac{\partial}{\partial x^\ell}}) = \frac{R_{1x^{\ell}x^{\ell}1}}{2g_\sigma}= 
-{\frac{1}{2}}D(|\mu_{\ell}|^2)(z_0).
\end{equation*}
Similarly, one can work out other curvatures:
\begin{equation*}
\frac{R_{1y^{\ell}y^{\ell}1}}{2g_\sigma}=\frac{R_{2x^{\ell}x^{\ell}2}}{2g_\sigma}=\frac{R_{2y^{\ell}y^{\ell}2}}{2g_\sigma}=
-{\frac{1}{2}}D(|\mu_{\ell}|^2)(z_0).
\end{equation*}
They are all bounded from above by $-\frac16 |\mu_{\ell}|^2(z_0)$, hence non-positive.
\end{proof}

\section{Application: The geometry of the {\WPg}}
An application of our method is to study the geometry of the {\WPm} in {\TS}, in particular, three-manifold formed by a 
surface bundle over a {\WPg} $\gamma$. This manifold is the union of germs $N_\sigma$ over $\gamma$, where 
$\sigma \in \gamma$, and the fiber at $\sigma$ is a hyperbolic surface in the conformal class determined by $\sigma$. We always 
assume the curve $\gamma$ is parametrized by its arc length.

Recall the {\Hd}s $\phi(t)dz^2$ of the family of {\hmp}s from $(\Sigma,\sigma)$ to $(\Sigma, g_{\rho(t)}dwd\bar{w})$ 
determine a curve $\rho(t) \in T_g$. It is crucial to us that, when $\phi(t)dz^2$ is a ray in $Q(\sigma)$, i.e., for some 
$\phi_0 dz^2 \in Q(\sigma)\backslash 0$, $\phi(t) = t\phi_0$, then the slice $\rho(t)$ is a {\WP} geodesic at $t=0$ 
(\cite {Ahl61}). This permits calculation of curvatures and the second fundamental form of the fibers of 
$N=\bigcup_\sigma N_\sigma$ by local computation on the germs.

This section is organized in subsections. \S 4.1 contains local calculations where we determine the {\Sc}s of the germ 
$N_\sigma$. The asymptotic behavior of these curvatures near the infinity is also investigated in the subsection; in 
\S 4.2, we study the {\sff} of the fiber hyperbolic surface of $N_\sigma$ at $\sigma$, and show that this fiber is 
minimal; in \S 4.3, we equip the {\sbd} $N=\bigcup_\sigma N_\sigma$ a Riemannian structure and prove the 
Theorem 1.4.

\subsection{Curvatures of the germ $N_\sigma$}
Let $\gamma$ be a {\WPg} arc in {\TS} that passes through $(\Sigma, \sigma)$ and in the direction of the harmonic 
{\Bd} $\mu_0(z){\frac{d\bar{z}}{dz}} =\frac{\bar{\phi}_0d\overline{z}^2}{g_\sigma dz d\overline{z}} \in B_h(\sigma)$, and 
$z = x+\sqrt{-1}y$. The {\WPg} arc $\gamma$ is parametrized by its arc length, so we have 
\begin{equation*}
\|\mu_0\|_{WP}=(\int_{(S,\sigma)}|\mu_0(z)|^2 dA(z))^{\frac12} =1.
\end{equation*}

In this subsection, we focus on local geometry near $\sigma$, i.e., we consider the germ 
$N_\sigma$ over the point $\sigma \in \gamma$, with the metric $H$ as in \eqref{H}:
\begin{equation*}
H = g_{\rho(w)}dwd\overline{w} + dt^2.
\end{equation*}
We obtain {\Sc}s of $N_\sigma$ with the metric $H$ for tangent vectors $\frac{\partial}{\partial x}$, 
$\frac{\partial}{\partial y}$ and $\frac{\partial}{\partial t}$, at $t=0$, and evaluating at 
$z_0 \in (\Sigma,\sigma)$, that is:

\begin{proof} [Proof of Theorem \ref{thm:g}] 

We set up similarly as in the proof of the Theorem \ref{thm:curve}.

For a family of {\hym}s $g_{\rho(w)}dwd\bar{w}$, we have a family of {\hqd}s 
$\phi(t)dz^2 = t\phi_0 dz^2 \in Q(\sigma)$ associated to the {\hmp}s $w(t)$ from $(\Sigma,\sigma)$ to 
$(\Sigma, g_{\rho(w)}dwd\bar{w})$.

The pullback metric on $(\Sigma,\sigma)$ is given by \eqref{pb}:
\begin{equation*}
w^{*} g_{\rho(t)}|dw(t)|^{2} = \phi(t) dz^{2} + g_{\sigma}e(t)|dz|^{2} + {\bar{\phi}}(t)d{\bar{z}}^{2}, 
\end{equation*}
where $e=e(t)$ is again the energy density of $w(t)$, and $\phi(t) =t\phi_0$.

The metric $H$ on the germ $N_\sigma$ is, in real coordinates, the following:
\begin{align*}\label{H-r}
&t\phi_0 dz^2+\sigma e(t)dz d\overline{z}+t\overline{\phi}_0d\overline{z}^2+dt^2  \\
=&\ t\phi_0 (dx^2-dy^2+2idxdy)+ g_\sigma e(t)(dx^2+dy^2)+ t\overline{\phi}_0(dx^2-dy^2-2idxdy)+dt^2 \\
=&\ (g_\sigma e+ 2t Re \phi_0)dx^2-4t Im\phi_0 dxdy+(g_\sigma e- 2t Re \phi_0)dy^2 +dt^2
\end{align*}

The pullback metric on the family of surfaces can be represented by its matrix form as follows:
\begin{equation}\label{fff}
(g_{ij}(t))=\left(
\begin{array}{ccc}
g_\sigma e(t)+ 2t Re \phi_0    & -2t Im\phi_0          \\
-2t Im\phi_0              &g_\sigma e(t)- 2t Re \phi_0 
\end{array}
\right).
\end{equation}
As before, we use indices $1$, $2$ for variables $x$ and $y$, respectively, to simplify the notation in the {\Cs}s. We 
now use index $3$ for the variable $t$ for the same purpose.

The values of relevant {\Cs}s, at $t=0$, can be computed as follows:
$$
\left\{
\begin{matrix}
\Gamma_{11}^{1}=\displaystyle \frac{(g_{\sigma})_1}{2g_{\sigma}}, 
& \Gamma_{11}^{2}= \displaystyle -\frac{(g_{\sigma})_2}{2g_\sigma},
& \Gamma_{11}^3= \displaystyle -Re\phi_0,  \\
\\
\Gamma_{12}^{1}=\displaystyle \frac{(g_{\sigma})_2}{2\sigma}, 
& \Gamma_{12}^{2} =\displaystyle \frac{(g_{\sigma})_1}{2\sigma},
& \Gamma_{12}^3 = \displaystyle Im\phi_0,\\
\\
\Gamma_{22}^{1}=\displaystyle -\frac{(g_{\sigma})_1}{2g_\sigma}, 
& \Gamma_{22}^{2}=\displaystyle \frac{(g_{\sigma})_2}{2g_\sigma}, 
& \Gamma_{22}^3=\displaystyle Re\phi_0,
\end{matrix}
\right.
$$
and
$$
\left\{
\begin{matrix}
\Gamma_{13}^{1}=\displaystyle \frac{Re\phi_0}{g_\sigma}, 
& \Gamma_{13}^{2}=\displaystyle -\frac{Im\phi_0}{g_\sigma}, 
& \Gamma_{13}^3= \displaystyle 0,  \\
\\
\Gamma_{23}^{1}=\displaystyle -\frac{Im\phi_0}{g_\sigma}, 
& \Gamma_{23}^{2} =\displaystyle -\frac{Re\phi_0}{g_\sigma},
& \Gamma_{23}^3 = \displaystyle 0,\\
\\
\Gamma_{33}^{1}=\displaystyle 0, 
& \Gamma_{33}^{2}=\displaystyle 0, 
& \Gamma_{33}^3=\displaystyle 0.
\end{matrix}
\right.
$$
The curvature tensor $R_{1221}$, at $t=0$, is
\begin{eqnarray*}
R_{1221}|_0 &=& g_\sigma ( \Gamma_{22,1}^{1} - \Gamma_{21,2}^{1} + 
\Gamma_{22}^{\beta}\Gamma_{\beta 1}^{1}-\Gamma_{21}^{\beta}\Gamma_{\beta2}^{1})|_0 \nonumber.
\end{eqnarray*}  

Applying the curvature on $(\Sigma,\sigma)$ is $-1$, and above values for {\Cs}s, we have:
\begin{eqnarray*}
R_{1221}|_0
&=& g_\sigma (-g_\sigma + \frac{(Re\phi_0)^2}{g_\sigma} + \frac{(Im\phi_0)^2}{g_\sigma}) \\
&=& -g_\sigma^2+ |\phi_0|^2.
\end{eqnarray*}    
So we find that, at $(\sigma, z_0)$, the curvature in the fiber directions is
\begin{equation*}
K(\frac{\partial}{\partial x}, \frac{\partial}{\partial y})(\sigma, z_0) = \frac{R_{1221}}{g_\sigma^2}= -1+|\mu_0(z_0)|^2.
\end{equation*}  
We are left to determine curvatures $K(\frac{\partial}{\partial x}, \frac{\partial}{\partial t})$ and 
$K(\frac{\partial}{\partial y}, \frac{\partial}{\partial t})$. 

The {\ct} $R_{1331}$ can be computed as follows:
\begin{equation*}
R_{1331}|_0 = g_\sigma ( \Gamma_{33,1}^{1} - \Gamma_{31,3}^{1} + 
\Gamma_{33}^{\beta}\Gamma_{\beta1}^{1}-\Gamma_{31}^{\beta}\Gamma_{\beta3}^{1})|_0,
\end{equation*}
where Einstein notation is employed for $\beta = 1, 2, 3$.

It is easy to verify from the values of the {\Cs}s at $t=0$, that 
\begin{eqnarray*}
\sum_{\beta=1}^3 (\Gamma_{33}^{\beta}\Gamma_{\beta1}^{1}-\Gamma_{31}^{\beta}\Gamma_{\beta3}^{1})|_0
&=& -(\Gamma_{31}^{1})^2|_0 - (\Gamma_{31}^2\Gamma_{23}^1)|_0\\
&=& -|\mu_0|^2,
\end{eqnarray*}
and we apply the second variation of $e(t)$ to find
\begin{eqnarray*}
\Gamma_{13,3}^{1}|_0
&=& \frac12 (-\frac{4(Re\phi_0)^2}{g_\sigma^2} + 2D(|\mu_0|^2) + 2|\mu_0|^2 -\frac{4(Im\phi_0)^2}{g_\sigma^2}) \\
&=& D(|\mu_0|^2) - |\mu_0|^2.
\end{eqnarray*}
Therefore we have 
\begin{eqnarray*}
R_{1331}|_0 &=& g_\sigma ( 0 - D(|\mu_0|^2) + |\mu_0|^2 -|\mu_0|^2) \\
&= & g_\sigma (-D(|\mu_0|^2)),
\end{eqnarray*}
and at $z_0$
\begin{equation*}
K(\frac{\partial}{\partial x}, \frac{\partial}{\partial t})(\sigma, z_0) = \frac{R_{1331}}{g_\sigma}= -D(|\mu_0|^2)(z_0) \le 0.
\end{equation*}  
The last curvature is calculated in the similar fashion and this completes the proof of the Theorem \ref{thm:g}.
\end{proof}

We are particularly interested in the case when $\sigma \in T_g$ travels along a {\WPg} $\gamma$ near the 
infinity of the augmented {\TS}, in which case, at least one simple closed curve on the surface $(\Sigma,\sigma)$ is 
being pinched, as the result, the norm $|\mu|$ is unbounded. Therefore, we find:

\begin{cor}\label{thm:de}
$Sup_{\sigma \in T_g} K(\frac{\partial}{\partial x}, \frac{\partial}{\partial y}) = +\infty$, and 
$Inf_{\sigma \in T_g}K(\frac{\partial}{\partial x}, \frac{\partial}{\partial t}) = -\infty$.
\end{cor}
\begin{proof}
When $\sigma \in T_g$ is near the infinity of the augmented {\TS}, a short essential curve on $(\Sigma,\sigma)$ 
is being pinched, and therefore $Sup_{z \in (\Sigma,\sigma)}|\mu(z)| = +\infty$. Corollary \ref{thm:de} is then the 
consequence of this and the Lemma \ref{thm:w}. 
\end{proof}

Naturally one hopes that $N_\sigma$ is negatively curved. Theorem 1.3 gives a negative answer to this. Therefore, 
as often seen in Riemannian geometry, it is natural to modify a given metric for better property. On the germ 
$N_\sigma$, based on the metric $H$ in \eqref{H}, we consider the modified metric $H_f$ as follows:
\begin{equation}
H_f = g_{\rho(w)}dwd\overline{w} + f(t)dt^2,
\end{equation}
where $f(t) > 0$ and $f(0) =1$. Proceeding as in the proof of Theorem \ref{thm:g}, it is easy to show the following:
\begin{pro}
The curvature $K(\frac{\partial}{\partial x}, \frac{\partial}{\partial y})$ of the {\WPg} $\gamma$ in the fiber directions, 
with respect to the metric $H_f$, at $t=0$, and $z=z_0$ is equal to $-1+|\mu_0|^2(z_0)$. Therefore 
$(N_\sigma, H_f)$ is not negatively curved.
\end{pro} 

\subsection{Minimality: the germ $N_\sigma$}
We now consider how the fiber hyperbolic surface $(\Sigma, g_\sigma |dz|^2)$ interacts with the 
germ $N_\sigma$. Naturally we study the {\sff}.

Since the $t$-direction is perpendicular to each surface, the {\sff} $(h_{ij}(t))$ can be calculated as 
$$
(h_{ij}(t))=\left(
\begin{array}{ccc}
g_\sigma e'(t)+ 2 Re \phi_0    & -2 Im\phi_0          \\
-2 Im\phi_0              &g_\sigma e'(t)- 2 Re \phi_0 
\end{array}
\right).
$$
Here $e'(t)$ is the $t-$derivative of the energy density $e(t)$ of the family of {\hmp}s $w(t)$.

Since $e(0)=1$ and $e'(0)=0$, when evaluated at $t=0$, we have 
\begin{equation}
(h_{ij}(0))=\left(
\begin{array}{ccc}
2 Re \phi_0    & -2 Im\phi_0          \\
-2 Im\phi_0              &- 2 Re \phi_0 
\end{array}
\right).
\end{equation}
Now the principal curvatures of surface at time $t=0$ are the eigenvalues of the following matrix:
\begin{equation*}
(g_{\sigma})^{-1}(h_{ij}(0))=
\frac{1}{g_\sigma}
\left(
\begin{array}{ccc}
2 Re \phi_0    & -2 Im\phi_0          \\
-2 Im\phi_0              &- 2 Re \phi_0 
\end{array}
\right).
\end{equation*}
Its two eigenvalues are now:
\begin{equation}
\lambda = \pm2\frac{|\phi_0|}{g_\sigma} =\pm2 |\mu_0|.
\end{equation}

Therefore we have proved:
\begin{theorem}\label{thm:minimal}
Each fiber hyperbolic surface $(S, g_\sigma dzd\bar{z})$ is a minimal surface of the germ $N_\sigma$.
\end{theorem} 
We note that the {\pc}s are unbounded if $\sigma$ is near the infinity of the augmented {\TS}.
\subsection{Minimality: surface bundle $N$ over $\gamma$}
Let $\gamma(t)$ be a {\WPg} arc, for $0 \leq t \leq T$, parametrized by its arc length. We denote the collection 
of germs $N_\sigma$ for $\sigma \in \gamma(t)$ by $N$, and it is clear that the three-manifold 
$N = \Sigma \times [0,T]$ is a surface bundle over $\gamma(t)$. In this subsection, we prove the Theorem 1.4.

\begin{proof} [proof of Theorem 1.4] We now equip $N$ with a Riemannian structure. We wish to equip $N$ with a 
metric of the form
\begin{equation}\label{M}
g_{\gamma(t)}dw(t)d\bar{w}(t) + dt^2
\end{equation}
where $g_{\gamma(t)}$ is the {\hym} in the conformal class of $\gamma(t)$.

Let $\Mcal_{-1}$ be the space of {\hym}s on a topological surface $S$, and recall that {\TS} is the space of {\hym}s 
on $S$, up to orientation-preserving diffeomorphisms in the homotopy class of the identity: $T_g = \Mcal_{-1}/D_0$, 
where $D_0$ is the identity component of the diffeomorphism group.

Let $[g_{\gamma(t)}]$ denote the fiber over $g_{\gamma(t)} \in T_g$. At any $g_1 \in [g_{\gamma(t)}]$, the {\ts} 
$T_{g_1}\Mcal_{-1}$ is identified as $Sym(0,2)$, the space of symmetric $(0,2)$-tensors on $S$. Let 
$h_1 \in Sym(0,2)$ be the symmetric $(0,2)$-tensor which induces a deformation of $g_1$ preserving the scalar 
curvature.  It is divergence-free and traceless, and moreover the deformation is smoothly dependent in $t$ by 
Theorem A of \cite{FM75}, or Theorems 2.4.2 of  \cite{Tro92}. In our case, $g_1$ is the hyperbolic metric 
$\sigma\in [g_{\gamma(t)}]$ and $h_1$ is the holomorphic quadratic differential in the tangent direction 
$\phi_0 dz^2$, and the smooth dependence results in a $C^\infty$ Riemannian metric on $N$ such that the germs 
at each fiber is identical to $N_\sigma$.

Note that the metric \eqref{M} agrees with metric \eqref{H} associated to $\rho(t)$ up to second order (\cite {Ahl61}), 
and now parts (1) and (2) follow from Theorem 4.3, and part (3) follows from the Theorem 1.3.
\end{proof}
By a theorem of Sullivan (\cite{Sul79}), any compact Riemannian manifold with a taut foliation admits a Riemannian 
metric such that each leaf of the foliation is a minimal surface. Theorem 1.4 shows that the $3$-manifold associated 
to a closed {\WPg} is such an example:
\begin{cor}
If $\gamma$ is a closed {\WPg} loop in moduli space $\mathcal{M}(S)$, then the metric \eqref{M} on the associated 
{\tm} $N$ is a Sullivan metric.
\end{cor} 
\begin{proof}
The {\WPg} loop lifts to a path of hyperbolic metrics $\gamma(t)$ in $\Mcal_{-1}$, as in the proof of Theorem 1.4, such 
that the hyperbolic surfaces $\Sigma_0=\Sigma\times\{0\}$ and $\Sigma_1=\Sigma\times \{T\}$ are isometric by an 
orientation-preserving mapping class $\psi:\Sigma_0\to \Sigma_1$. We glue the boundary of $\Sigma \times [0,T]$ 
equipped with the metric \eqref{M} by $\psi$ to obtain a surface bundle $N$. Since the normal direction is 
$\partial/ \partial t$, the normal vectors $\partial/\partial  t$ on $\Sigma_0$ match up with the normal vectors on 
$\Sigma_1$, and the resulting metric on $N$ is smooth. The property of all the fibers being minimal follows from 
Theorem 1.4.
\end{proof}
\begin{rem}
\noindent
\begin{enumerate}
\item
The minimality of the germs does not require $\gamma \subset T_g$ to be a geodesic. It is essentially due to the 
property that deformations $h_1$ can be chosen as traceless. Being a {\WPg} allows us to apply the method of 
{\hmp}s to calculate the curvatures of $N$.
\item
Each fiber of $N$ has the same surface area. It is possible to provide an expression for the associated calibration 
$\omega$, a $2$-form of co-mass $1$, on $N$. It is the hyperbolic area form when it is restricted on each fiber and 
it is closed since the first variation of the area along a {\WPg} vanishes (\cite{Ahl61},\cite{Wol86}).
\end{enumerate}
\end{rem}


\begin{thebibliography}{GHW09}

\bibitem[Ahl61]{Ahl61}
Lars~V. Ahlfors, \emph{Some remarks on {T}eichm\"uller's space of {\RS}s}, Ann.
  of Math. (2) \textbf{74} (1961), 171--191.

\bibitem[Ber72]{Ber72}
Lipman Bers, \emph{Uniformization, moduli, and {K}leinian groups}, Bull. London
  Math. Soc. \textbf{4} (1972), 257--300.

\bibitem[Bro03]{Bro03}
Jeffrey~F. Brock, \emph{The {\WPm} and volumes of 3-dimensional hyperbolic
  convex cores}, J. Amer. Math. Soc. \textbf{16} (2003), no.~3, 495--535.

\bibitem[Chu76]{Chu76}
Tienchen Chu, \emph{The {\WPm} in the moduli space}, Chinese J. Math.
  \textbf{4} (1976), no.~2, 29--51.

\bibitem[ES64]{ES64}
James Eells, Jr. and J.~H. Sampson, \emph{Harmonic mappings of {R}iemannian
  manifolds}, Amer. J. Math. \textbf{86} (1964), 109--160.
  
\bibitem[FM75]{FM75}
Arthur Fischer and Jerrold E. Marsden, \emph{Deformations of the scalar curvature}, Duke 
Math. J. \textbf{43} (1975), 519--547.  

\bibitem[GHW09]{GHW09}
Ren Guo, Zheng Huang, and Biao Wang, \emph{Quasi-fuchsian three-manifolds and
  metrics on {\TS}}, to appear, Asian J. Math., arXiv:0909.2426v2 (2009).

\bibitem[Hua05]{Hua05}
Zheng Huang, \emph{Asymptotic flatness of the {\WPm} on {\TS}}, Geom. Dedicata
  \textbf{110} (2005), 81--102.

\bibitem[Hua07]{Hua07}
\bysame, \emph{On asymptotic {\WP} geometry of {\TS} of {\RS}s}, Asian J. Math.
  \textbf{11} (2007), no.~3, 459--484.

\bibitem[Jos91a]{Jos91a}
J{\"u}rgen Jost, \emph{Harmonic maps and curvature computations in {\Tt}}, Ann. 
Acad. Sci. Fenn. Ser. A I Math. \textbf{16} (1991), no.~1, 13--46.

\bibitem[Jos91b]{Jos91b}
\bysame, \emph{Two-dimensional geometric variational problems}, Pure and
  Applied Mathematics (New York), John Wiley \& Sons Ltd., 1991.
  
  \bibitem[JY09]{JY09}
J{\"u}rgen Jost and Shing~Tung Yau, \emph{Harmonic mappings and moduli 
spaces of {\RS}s}, Preprint, (2009).

\bibitem[Mas76]{Mas76}
Howard Masur, \emph{Extension of the {\WPm} to the boundary of {\TS}}, Duke
  Math. J. \textbf{43} (1976), no.~3, 623--635.

\bibitem[Min92a]{Min92b}
Yair~N. Minsky, \emph{Harmonic maps into hyperbolic 3-manifolds}, Trans. Amer.
  Math. Soc. \textbf{332} (1992), no.~2, 607--632.

\bibitem[Min92b]{Min92a}
\bysame, \emph{Harmonic maps, length, and energy in {\TS}}, J. Differential
  Geom. \textbf{35} (1992), no.~1, 151--217.

\bibitem[Sam78]{Sam78}
J.~H. Sampson, \emph{Some properties and applications of harmonic mappings},
  Ann. Sci. \'Ecole Norm. Sup. (4) \textbf{11} (1978), no.~2, 211--228.

\bibitem[Sch86]{Sch86}
Georg Schumacher, \emph{Harmonic maps of the {\ms} of compact {\RS}s}, Math.
  Ann. \textbf{275} (1986), no.~3, 455--466.
 
\bibitem[Sul79]{Sul79}
Dennis Sullivan, \emph{A homological characterization of foliations consisting of minimal surfaces}, 
Comment. Math. Helv. \textbf{54} (1979), no.~2, 218--223. 

\bibitem[SY78]{SY78}
Richard Schoen and Shing~Tung Yau, \emph{On univalent {\hmp}s between
  surfaces}, Invent. Math. \textbf{44} (1978), no.~3, 265--278.

\bibitem[Tro86]{Tro86}
Anthony J. Tromba, \emph{On a natural algebraic affine connection on the space of
  almost complex structures and the curvature of {\TS} with respect to its
  {\WPm}}, Manuscripta Math. \textbf{56} (1986), no.~4, 475--497.
  
  \bibitem[Tro92]{Tro92}
\bysame, \emph{{\Tt} in Riemannian geometry}, Lectures in Math. Birkh\"{a}user Verlag, 1992.

\bibitem[Wol75]{Wol75}
Scott Wolpert, \emph{Noncompleteness of the {\WPm} for {\TS}}, Pacific J. Math.
  \textbf{61} (1975), no.~2, 573--577.

\bibitem[Wol86]{Wol86}
\bysame, \emph{Chern forms and the {R}iemann tensor for the {\ms} of
  curves}, Invent. Math. \textbf{85} (1986), no.~1, 119--145.

\bibitem[Wol87]{Wol87}
\bysame, \emph{Geodesic length functions and the {N}ielsen problem}, J.
  Differential Geom. \textbf{25} (1987), no.~2, 275--296.

\bibitem[Wol89]{Wol89}
Michael Wolf, \emph{The {T}eichm\"uller theory of harmonic maps}, J.
  Differential Geom. \textbf{29} (1989), no.~2, 449--479.

\bibitem[Wol03]{Wol03}
Scott~A. Wolpert, \emph{Geometry of the {\WP} completion of {\TS}}, Surveys in
  differential geometry, {V}ol.\ {VIII} ({B}oston, {MA}, 2002), Surv. Differ.
  Geom., VIII, Int. Press, Somerville, MA, 2003, pp.~357--393.

\bibitem[Wol06]{Wol06}
\bysame, \emph{Weil-{P}etersson perspectives}, Problems on mapping class groups
  and related topics, Proc. Sympos. Pure Math., vol.~74, Amer. Math. Soc.,
  Providence, RI, 2006, pp.~269--282.

\bibitem[Wol08]{Wol08}
Michael Wolf, \emph{The {\WP} hessian of length on {\TS}}, preprint (2008).

\bibitem[Wol09]{Wol09}
Scott~A. Wolpert, \emph{The {\WPm} geometry}, Handbook of {\Tt}. {V}ol. {II},
  IRMA Lect. Math. Theor. Phys., vol.~13, 2009, pp.~47--64.

\end{thebibliography}
\end{document}